\documentclass[12pt]{article}
\usepackage[margin=1in]{geometry}
\usepackage{amsmath, amssymb, amsthm}
\usepackage{hyperref}

\theoremstyle{plain}
\newtheorem{theorem}{Theorem}[section]
\newtheorem{lemma}[theorem]{Lemma}
\newtheorem{proposition}[theorem]{Proposition}
\theoremstyle{definition}
\newtheorem{definition}[theorem]{Definition}
\newtheorem{remark}[theorem]{Remark}

\newcommand{\dvr}{d}
\DeclareMathOperator{\vp}{v}

\title{A Proof of Bala's Congruence Conjecture for A028342}
\author{Ahaan Kallat\\
Khoury College of Computer Sciences\\
Northeastern University\\
Boston, MA, USA\\
\texttt{kallat.a@northeastern.edu}}
\date{\today}

\begin{document}
\maketitle

\begin{abstract}
Let $a(n)$ be the sequence A028342 in the On-Line Encyclopedia of Integer
Sequences (OEIS), defined by the exponential generating function
$\sum_{n\ge0} a(n)x^n/n! = \prod_{i\ge1}(1-x^i)^{-1/i}$. Equivalently, $a(n)$
counts permutations of an $n$-element labeled set in which every cycle is
assigned one divisor of its length, where a cycle of length $m$ has
$\dvr(m)$ choices, $\dvr(m)$ being the number of positive divisors of $m$. We prove
a family of congruences for $a$, conjectured by Peter Bala. They state that
$k \mid a(n+k)+a(n)$ for odd $k$, that $k \mid a(n+k)-a(n)$ for
$k\equiv 0,2,6 \pmod 8$, and that $k \mid 2\bigl(a(n+k)-a(n)\bigr)$ for
$k\equiv 4\pmod 8$. The proof first establishes a product congruence
$a(n+k)\equiv a(n)a(k)\pmod k$, and then computes $a(p^r)\bmod p^r$ for each
prime power by counting the colored permutations fixed by a subgroup of order
$p$.
\end{abstract}



\section{Introduction}

The On-Line Encyclopedia of Integer Sequences~\cite{oeis} records many
empirical observations contributed by its users and labeled ``Conjecture.''
Some of these are congruences guessed from numerical data but left unproved.
This note gives a proof of one such conjecture, attributed to Peter Bala, for the
sequence A028342~\cite{oeisA028342}.

The sequence $a(n)$ is defined by the exponential generating function
\begin{equation}\label{eq:egf}
\sum_{n\ge0} a(n)\,\frac{x^n}{n!}
  = \prod_{i\ge1}(1-x^i)^{-1/i}
  = \exp\!\left(\sum_{m\ge1}\frac{\dvr(m)\,x^m}{m}\right),
\end{equation}
where $\dvr(m)$ is the number of positive divisors of $m$. All generating-function
identities in this note, including~\eqref{eq:egf} and~\eqref{eq:combdef} below, are
identities of formal power series in $\mathbb{Q}[[x]]$: the infinite product and the
exponentials are defined termwise in the usual way (each coefficient of $x^n$ receives
contributions from only finitely many factors), and no question of analytic convergence
arises. The second equality in~\eqref{eq:egf}
follows by taking logarithms. Using $-\log(1-x^i) = \sum_{j=1}^{\infty} x^{ij}/j$,
\[
\log\prod_{i\ge1}(1-x^i)^{-1/i}
  = \sum_{i=1}^{\infty}\frac1i\sum_{j=1}^{\infty}\frac{x^{ij}}{j}
  = \sum_{i=1}^{\infty}\sum_{j=1}^{\infty}\frac{x^{ij}}{ij},
\]
and collecting the terms with $ij=m$ gives the coefficient
\[
\sum_{\substack{i,j\ge1\\ ij=m}}\frac{1}{ij}
  = \frac1m\sum_{\substack{i,j\ge1\\ ij=m}}1
  = \frac{\dvr(m)}{m},
\]
which is the exponent of $x^m$ in~\eqref{eq:egf}.

By the exponential formula, $a(n)$ has the following combinatorial description,
which was recorded for A028342 by Ricardo Gómez Aíza~\cite{gomezAiza}.
See also the fuller treatment in Section~3.3.1 of~\cite{ahmadi2024}.
A \emph{colored permutation} is a pair $(\sigma, c)$, where $\sigma$ is a
permutation of an $n$-element labeled set and $c$ assigns to each cycle of
$\sigma$ one divisor of that cycle's length. The choices are made independently
across cycles, so a cycle of length $m$ has $\dvr(m)$ possible colors. For
instance, a cycle of length $4$ has three possible colors, corresponding to the
divisors $1, 2, 4$. Then $a(n)$ is the number of colored permutations, which
equals
\begin{equation}\label{eq:combdef}
a(n) = \sum_{\sigma}\ \prod_{\text{cycles } C \text{ of } \sigma}
       \dvr(|C|).
\end{equation}
The product over the cycles of a single $\sigma$ counts the colorings of that
permutation, and the sum ranges over all $n!$ permutations. The first values are
$a(0),a(1),\dots = 1, 1, 3, 11, 59, 339, 2629, \dots$.

The same sequence also appears in the unified framework of Ahmadi, Gómez-Aíza,
and Ward~\cite{ahmadi2024}, who study divisor-controlled families of colored
combinatorial objects and prove, among other results, the asymptotic behavior
of A028342. The present note addresses a different question: it proves Bala's
congruence conjecture for A028342, rather than an asymptotic formula.

We prove the following.

\begin{theorem}\label{thm:main}
For all $n\ge0$ and $k\ge1$,
\[
\begin{cases}
k\mid a(n+k)+a(n), & k\text{ odd},\\[1mm]
k\mid a(n+k)-a(n), & k\equiv0,2,6\pmod 8,\\[1mm]
k\mid 2\bigl(a(n+k)-a(n)\bigr), & k\equiv4\pmod 8.
\end{cases}
\]
\end{theorem}
The three cases are exhaustive, since every $k\ge1$ is odd, congruent to
$4\pmod 8$, or congruent to $0$, $2$, or $6\pmod 8$.

The proof has one main idea. A group action lets us discard the objects lying in
orbits whose size is exactly the modulus, since each such orbit contributes $0$
modulo that modulus. We use this idea twice. In Section~\ref{sec:product} the
mixed objects fall into $C_k$-orbits of size $k$, so modulo $k$ only the unmixed
count $a(n)a(k)$ remains. In Section~\ref{sec:primepower} the objects with
$C_{p^r}$-orbits of size $p^r$ vanish modulo $p^r$, leaving a count of fixed
points. Between these two steps, Section~\ref{sec:reduction} uses the first
congruence to reduce $a(k)\bmod k$ to the prime-power values $a(p^r)\bmod p^r$ by
the Chinese remainder theorem, and Section~\ref{sec:completing} combines those
values into the three cases of Theorem~\ref{thm:main}.

We have checked Theorem~\ref{thm:main} and the intermediate identities directly
for all $k$ up to $40$ and all $n$ up to $30$.\footnote{The values of $a(n)$
were generated from
\[
a(n)=\sum_{m=1}^{n}\binom{n-1}{m-1}(m-1)!\,\dvr(m)\,a(n-m),
\qquad a(0)=1,
\]
which follows by differentiating~\eqref{eq:egf}.}

\section{A product congruence}\label{sec:product}

Throughout this section $A$ and $B$ are disjoint label sets with $|A|=n$ and
$|B|=k$.

\begin{definition}\label{def:objects}
An \emph{object} is a colored permutation of $A\cup B$. That is, it is a
permutation of $A\cup B$ together with a choice of one divisor of the length of
each of its cycles. An object is \emph{unmixed} if every cycle lies entirely in
$A$ or entirely in $B$, and \emph{mixed} if some cycle contains both an
$A$-label and a $B$-label. Let $X$ denote the set of all objects, and let
$X_{\mathrm{unmixed}}$ and $X_{\mathrm{mixed}}$ denote the two classes. Then
$X = X_{\mathrm{unmixed}} \sqcup X_{\mathrm{mixed}}$ and $|X| = a(n+k)$.
\end{definition}

We use two standard notions from group actions. When a finite group $G$ acts on
a set, the \emph{orbit} of a point is the set of points reachable from it by
applying elements of $G$, and the \emph{stabilizer} of a point is the subgroup
of elements that fix it. The orbit--stabilizer theorem states that
\begin{equation}\label{eq:orbstab}
|\text{orbit of }x|\cdot|\text{stabilizer of }x| = |G|.
\end{equation}
In particular, a point with trivial stabilizer has an orbit of size $|G|$. We
use~\eqref{eq:orbstab} here and again in Section~\ref{sec:primepower}.

\begin{proposition}\label{prop:product}
For all $n\ge0$ and $k\ge1$,
\[
a(n+k)\equiv a(n)\,a(k)\pmod{k}.
\]
\end{proposition}

\begin{proof}
Let the cyclic group $C_k$ of order $k$ act on $X$ by cyclically rotating the
labels of $B$ among themselves and fixing every label of $A$. The color of each
cycle is carried along with the cycle under this relabeling, so the action
sends colored permutations to colored permutations. This action also preserves
the unmixed and mixed classes, since rotating $B$ neither creates nor destroys a
cycle that contains both an $A$-label and a $B$-label. We count $|X| = a(n+k)$
modulo $k$ by treating the two classes of Definition~\ref{def:objects}
separately.

\textit{Unmixed objects}: An unmixed object is a pair consisting of a colored
permutation of $A$ and a colored permutation of $B$, chosen independently.
Therefore
\[
|X_{\mathrm{unmixed}}| = a(n)\,a(k).
\]
This is an exact count, and it is all we use about the unmixed class.

\textit{Mixed objects}: We claim that every mixed object has trivial
$C_k$-stabilizer. By~\eqref{eq:orbstab} its orbit then has size $|C_k| = k$, so
$X_{\mathrm{mixed}}$ is a union of orbits each of size $k$, and hence
$|X_{\mathrm{mixed}}|$ is a multiple of $k$. To prove the claim, suppose some
nontrivial rotation $g\in C_k$ fixes a mixed object, and write $\sigma$ for its
underlying permutation, so that
$g\sigma g^{-1}=\sigma$, which is the same as $g\sigma=\sigma g$. Since $\sigma$
is mixed, some cycle contains both $A$-labels and $B$-labels. Traversing that
cycle, we may choose an $A$-label $u$ whose image $\sigma(u)$ is a $B$-label,
because the cycle must make at least one transition from an $A$-label to a
$B$-label. Since $g$ fixes every $A$-label, $g(u)=u$, and therefore
\[
g\bigl(\sigma(u)\bigr) = (g\sigma)(u) = (\sigma g)(u) = \sigma\bigl(g(u)\bigr)
  = \sigma(u).
\]
This says that the $B$-label $\sigma(u)$ is fixed by the nontrivial rotation
$g$ of $B$. But $C_k$ acts freely on $B$, so no nontrivial rotation can fix a
$B$-label. This contradiction shows that no nontrivial element of $C_k$ fixes a
mixed object. Hence the stabilizer is trivial, which proves the claim.

\textit{Conclusion}: The mixed objects contribute $0\pmod k$, and the unmixed
objects contribute $a(n)a(k)$. Since $|X|=a(n+k)$,
\[
a(n+k)\equiv a(n)\,a(k)\pmod{k}.\qedhere
\]
\end{proof}

\section{Reduction to prime powers}\label{sec:reduction}

The product congruence turns a relation between far-apart terms into a
multiplicative one, so the value of $a(k)\bmod k$ is now the only unknown. The
next lemma expresses that value through prime-power data.

\begin{lemma}\label{lem:reduction}
Let $q\mid k$ and write $k=sq$. Then
\[
a(k)\equiv a(q)^{\,k/q}\pmod q.
\]
\end{lemma}

\begin{proof}
Proposition~\ref{prop:product}, applied with $q$ in place of $k$, gives
$a(m+q)\equiv a(m)\,a(q)\pmod q$ for every $m\ge0$. Taking $m=q, 2q, 3q, \dots$
in turn and using this relation at each step,
\begin{align*}
a(2q) &\equiv a(q)\,a(q) = a(q)^2 \pmod q,\\
a(3q) &\equiv a(2q)\,a(q) \equiv a(q)^3 \pmod q,\\
       &\ \ \vdots\\
a(sq) &\equiv a((s-1)q)\,a(q) \equiv a(q)^{s} \pmod q.
\end{align*}
Since $sq=k$, the exponent equals $s=k/q$, which gives
$a(k)\equiv a(q)^{\,k/q}\pmod q$.
\end{proof}

Now fix $k$ and let $q_1, \dots, q_t$ be the maximal prime-power divisors of
$k$, where $q_i = p_i^{r_i}$ for distinct primes $p_i$. These are pairwise
coprime and their product is $k$. The Chinese remainder theorem states that an
integer is determined modulo $k$ by its residues modulo each $q_i$. By
Lemma~\ref{lem:reduction}, the residue of $a(k)$ modulo $q_i$ is determined by
$a(q_i)\bmod q_i$. It therefore suffices to determine $a(p^r)\bmod p^r$ for each
prime power $p^r$, which is the goal of the next section.

\section{Prime-power residues}\label{sec:primepower}

Throughout this section $\vp_p$ denotes the $p$-adic valuation on $\mathbb{Q}^{\times}$: for a
nonzero integer $x$, $\vp_p(x)$ is the exponent of the highest power of $p$ dividing $x$, and this
extends to nonzero rationals by $\vp_p(a/b) = \vp_p(a) - \vp_p(b)$ for nonzero integers $a,b$
(well defined since $\vp_p(a c)=\vp_p(a)+\vp_p(c)$ for the integers this note uses). We apply
$\vp_p$ below to a few explicit rationals, such as $W_\ell/\ell$.

\begin{proposition}\label{prop:primepower}
For an odd prime $p$ and $r\ge1$,
\[
a(p^r)\equiv -1\pmod{p^r}.
\]
For $p=2$,
\[
a(2)\equiv 1\pmod 2,\qquad
a(4)\equiv 3\pmod 4,\qquad
a(2^r)\equiv 1\pmod{2^r}\ \ (r\ge3).
\]
\end{proposition}

\begin{proof}
The proof occupies the rest of this section, split into a common Setup and then the odd-prime
and power-of-two cases.

\subsection{Setup}

The plan follows Section~\ref{sec:product}. The proof uses a cyclic group action. Objects lying in full-size orbits contribute multiples of the relevant prime power and therefore vanish modulo that prime power. The congruence is therefore reduced to counting the objects fixed by the subgroup $H$. Counting those fixed objects gives a generating function, and we read
off the coefficient we need by bounding a $p$-adic valuation.

Fix a prime power $N=p^r$ and let $X_N$ be the set of colored permutations on
$N$ labels, so $|X_N|=a(N)$. Let $C_N$ act on the $N$ labels by a full cyclic
rotation, and let $H$ be the unique subgroup of $C_N$ of order $p$. The orbits
of $H$ on the labels partition the labels into $M=p^{r-1}$ blocks of size $p$.

Every nontrivial subgroup of the cyclic group $C_N$ contains $H$. Therefore any
object with nontrivial $C_N$-stabilizer is already fixed by $H$, while any
object not fixed by $H$ has trivial $C_N$-stabilizer and, by the
orbit--stabilizer theorem~\eqref{eq:orbstab}, a full $C_N$-orbit of size
$N=p^r$. Each full orbit contributes $0\pmod{p^r}$ to $|X_N|=a(N)$, so only the
objects fixed by $H$ matter modulo $p^r$. Writing $X_N^H$ for the set of those
objects,
\begin{equation}\label{eq:fixedcount}
a(p^r)\equiv |X_N^{H}|\pmod{p^r}.
\end{equation}

We count the objects in $X_N^H$ through the permutation they induce on the
blocks. An object fixed by $H$ sends whole blocks to whole blocks. Indeed, if
\(y=hx\) lies in the same \(H\)-orbit as \(x\), then
\[
\sigma(y)=\sigma(hx)=h\sigma(x),
\]
because the object is fixed by \(H\). Thus \(\sigma(x)\) and \(\sigma(y)\) lie
in the same \(H\)-orbit, so \(\sigma\) sends each \(H\)-block onto an
\(H\)-block. It therefore induces a permutation of the $M$ blocks, and the
cycles of that permutation are what we call \emph{block-cycles}. A block-cycle
has length $\ell$ if it contains $\ell$ of the $H$-blocks.

Fix a block-cycle of length $\ell$. An object fixed by $H$ gives rise to this
block-cycle in one of two ways, according to the total $H$-shift accumulated on
going around it once. Concretely, identify each $H$-orbit of labels with
$\mathbb{Z}/p\mathbb{Z}$. Then each step from one block to the next in the
block-cycle carries an associated shift in $\mathbb{Z}/p\mathbb{Z}$, and the
total $H$-shift is the sum of these shifts around the block-cycle.
\begin{itemize}
\item \textit{Total shift zero}: The labels above the block-cycle form $p$
cycles of length $\ell$, permuted cyclically by $H$. For the coloring to be
fixed by $H$, all $p$ cycles must share one color, which gives $\dvr(\ell)$
choices. There are $p^{\ell-1}$ shift assignments with total shift zero.
\item \textit{Total shift nonzero}: The labels above the block-cycle form a
single cycle of length $p\ell$, which gives $\dvr(p\ell)$ choices. There are
$(p-1)p^{\ell-1}$ shift assignments with nonzero total shift.
\end{itemize}
Adding the two contributions, the effective weight of a block-cycle of length
$\ell$ is
\begin{equation}\label{eq:weight}
W_\ell = p^{\ell-1}\bigl(\dvr(\ell)+(p-1)\dvr(p\ell)\bigr).
\end{equation}

We now make explicit the bijection that turns this weight into the generating function
\eqref{eq:Xfixed}. An object fixed by $H$ is uniquely determined by: (i) the permutation
$\tau$ it induces on the $M$ blocks; and (ii), independently for each block-cycle $C$ of
$\tau$, a shift assignment (one of the $p^{|C|}$ sequences of $H$-shifts described above,
recording each step's shift as we go around $C$) together with the coloring it forces (one
shared color for the $p$ length-$|C|$ cycles if the total shift is zero, or one color for the
single length-$p|C|$ cycle if the total shift is nonzero). Different block-cycles impose no
compatibility constraints on one another, since they involve disjoint sets of labels, so these
choices are independent across the cycles of $\tau$.

Consequently, for a \emph{fixed}
permutation $\tau$ of the $M$ blocks, the $H$-fixed objects inducing $\tau$ have total weight
$\prod_{\text{cycles }C\text{ of }\tau} W_{|C|}$, since summing the shift-and-coloring weight
over one block-cycle $C$ gives exactly $W_{|C|}$ by~\eqref{eq:weight}. Summing over all $M!$
permutations $\tau$ of the blocks,
\[
|X_N^{H}| = \sum_{\tau}\ \prod_{\text{cycles }C\text{ of }\tau} W_{|C|},
\]
the sum ranging over all permutations $\tau$ of the $M$ blocks (as in~\eqref{eq:combdef}).
This is the same shape as~\eqref{eq:combdef}, with the weight function $\ell\mapsto W_\ell$ in
place of $\ell\mapsto\dvr(\ell)$, so the same exponential formula used in the Introduction to
pass from~\eqref{eq:combdef} to~\eqref{eq:egf} gives
\begin{equation}\label{eq:Xfixed}
|X_N^{H}| = M!\,[x^{M}]\exp\!\left(\sum_{\ell\ge1} W_\ell\,\frac{x^\ell}{\ell}\right).
\end{equation}
In each case below we read off this coefficient modulo $p^r$. The factors
$p^{\ell-1}$ in~\eqref{eq:weight} make this possible, because longer
block-cycles carry more powers of $p$ and drop out.

\subsection{The odd prime case}

Let $p$ be odd. From~\eqref{eq:weight}, using $\dvr(1)=1$ and $\dvr(p)=2$,
\[
W_1 = \dvr(1) + (p-1)\dvr(p) = 1 + 2(p-1) = 2p-1,
\]
and for $\ell\ge2$ the factor $p^{\ell-1}$ shows that $p^{\ell-1}\mid W_\ell$.
We show that every term of~\eqref{eq:Xfixed} that uses a block-cycle of length
$\ell\ge2$ is divisible by $p^r$, so that only block fixed points survive.

We write such a term explicitly. Fix a term in the expansion of~\eqref{eq:Xfixed}, let
$m_\ell$ denote the number of block-cycles of length $\ell\ge2$ it uses, and put
\[
L=\sum_{\ell\ge2}\ell\,m_\ell,\qquad s=\sum_{\ell\ge2}m_\ell,
\]
so $s\ge1$ and $L\ge2s$. Expanding $\exp\bigl(\sum_{\ell\ge1}W_\ell x^\ell/\ell\bigr)$ as the
product over $\ell$ of the exponential series in $W_\ell x^\ell/\ell$, and extracting the
coefficient of $x^M$ (multiplied by $M!$, as in~\eqref{eq:Xfixed}), the term corresponding to
this choice of $(m_\ell)_{\ell\ge2}$ is
\begin{equation}\label{eq:term}
\frac{M!}{(M-L)!}\,W_1^{\,M-L}\prod_{\ell\ge2}\frac{(W_\ell/\ell)^{m_\ell}}{m_\ell!}.
\end{equation}
Here $M!/(M-L)!=M(M-1)\cdots(M-L+1)$ is the block-placement factor (choosing and ordering
the $L$ blocks used by the length-$\ge2$ cycles among the $M$ blocks); the remaining $(M-L)!$,
paired with $W_1^{M-L}$, accounts for the $M-L$ length-$1$ block-cycles filling the unused
blocks; each factor of $\ell$ inside $W_\ell/\ell$ is the usual cyclic-ordering denominator
already built into the exponential formula; and $\prod_{\ell\ge2}m_\ell!$ is the repeated-length
denominator, since the $m_\ell$ block-cycles of a given length $\ell$ are interchangeable.
The $p$-adic valuation of~\eqref{eq:term} is at least the sum of three quantities.
\begin{itemize}
\item \textit{Block placement}: The factor $M(M-1)\cdots(M-L+1)$ has valuation
\[
\vp_p\bigl(M(M-1)\cdots(M-L+1)\bigr) = (r-1) + \vp_p\bigl((L-1)!\bigr).
\]
Since $L\le M=p^{r-1}$, among the $L$ consecutive integers
$M, M-1, \dots, M-L+1$, only $M$ is divisible by $p^{r-1}$, and it contributes
$\vp_p(M)=r-1$. The remaining $L-1$ integers $M-1, \dots, M-L+1$ contribute the
same powers of $p$ as $1, 2, \dots, L-1$, namely $\vp_p((L-1)!)$.
\item \textit{Cycle weights}: For each $\ell\ge2$ with $m_\ell>0$, $W_\ell/\ell$ has valuation at
least $(\ell-1)-\vp_p(\ell)$, because $p^{\ell-1}\mid W_\ell$ and dividing by $\ell$ removes at
most $\vp_p(\ell)$ factors of $p$. Summed with multiplicity $m_\ell$ over all $\ell\ge2$, listing
the lengths of the $s$ cycles as $\ell_1,\dots,\ell_s$ (so each $\ell$ with $m_\ell>0$ appears
$m_\ell$ times among the $\ell_i$), these contribute at least
$\sum_{i=1}^s(\ell_i-1-\vp_p(\ell_i))$.
\item \textit{Repeated lengths}: The denominator $\prod_{\ell\ge2}m_\ell!$ divides $s!$, since
$s!/\prod_{\ell\ge2}m_\ell!$ is the multinomial coefficient counting the ways to partition $s$
labeled cycles into groups of sizes $(m_\ell)_{\ell\ge2}$ by length, hence a positive integer.
Consequently $\sum_{\ell\ge2}\vp_p(m_\ell!)\le\vp_p(s!)$, so dividing by $\prod_{\ell\ge2}m_\ell!$
costs at most $\vp_p(s!)$.
\end{itemize}
Hence the total valuation is at least
\[
(r-1) + \vp_p\bigl((L-1)!\bigr)
       + \sum_{i=1}^s\bigl(\ell_i-1-\vp_p(\ell_i)\bigr)
       - \vp_p(s!).
\]
Since $p$ is odd and each $\ell_i\ge2$, we have
\[\ell_i-1-\vp_p(\ell_i)\ge1.\]
Indeed, set \(v=\vp_p(\ell_i)\). If \(v=0\), this is immediate from
\(\ell_i\ge2\). If \(v\ge1\), then
\(\ell_i\ge p^v\ge3^v\ge v+2\), so again \(v\le\ell_i-2\).
Dropping the nonnegative term
\(\vp_p((L-1)!)\) and using
\(\sum_i(\ell_i-1-\vp_p(\ell_i))\ge s\), the valuation is at least
\[
(r-1) + s - \vp_p(s!) \ge (r-1) + 1 = r.
\]
Here $s-\vp_p(s!)\ge1$, because $s!$ is not divisible by $p^{s}$ and so
$\vp_p(s!)\le s-1$. Thus every such term vanishes modulo $p^r$.
The remaining $M-L$ block-cycles of length $1$, if any, each have weight
$W_1=2p-1$, a $p$-adic unit, and therefore contribute nonnegative valuation.

Only the term with all block-cycles of length $1$ remains, and it contributes
$W_1^{M}$. Therefore
\[
a(p^r)\equiv W_1^{M} = (2p-1)^{p^{r-1}}\pmod{p^r}.
\]
Write $2p-1=-(1-2p)$. Since $p^{r-1}$ is odd, the overall sign is negative, so
\[
(2p-1)^{p^{r-1}} = -(1-2p)^{p^{r-1}} \equiv -1 \pmod{p^r}.
\]
Here $(1-2p)^{p^{r-1}}\equiv1\pmod{p^r}$ holds because $1-2p\equiv1\pmod p$, and
the lifting-the-exponent lemma then gives
$\vp_p\bigl((1-2p)^{p^{r-1}}-1\bigr) = \vp_p(-2p) + \vp_p(p^{r-1}) = 1 + (r-1)
= r$, which is valid since $p$ is odd. This proves $a(p^r)\equiv-1\pmod{p^r}$.

\subsection{The power of two case}

Let $p=2$, $N=2^r$, and $M=2^{r-1}$. From~\eqref{eq:weight},
\[
W_\ell = 2^{\ell-1}\bigl(\dvr(\ell)+\dvr(2\ell)\bigr),
\]
so in particular, using $\dvr(1)=1$, $\dvr(2)=2$, and $\dvr(4)=3$,
\[
W_1 = \dvr(1)+\dvr(2) = 1+2 = 3,\qquad
W_2 = 2\bigl(\dvr(2)+\dvr(4)\bigr) = 2(2+3) = 10,
\]
and $2^{\ell-1}\mid W_\ell$ for $\ell\ge3$.

\begin{lemma}\label{lem:two-elimination}
Let $r\ge3$ and $M=2^{r-1}$. In the expansion of
\[
M!\,[x^{M}]\exp\!\left(\sum_{\ell\ge1} W_\ell\,\frac{x^\ell}{\ell}\right),
\]
every term involving at least one block-cycle of length $\ell\ge3$ is divisible
by $2^r$.
\end{lemma}

\begin{proof}
As in the odd case, fix a term of the expansion, let $m_\ell$ denote the number of
block-cycles of length $\ell\ge3$ it uses, let $t=m_2\ge0$ be the number of block-cycles of
length $2$, and put
\[
L=\sum_{\ell\ge3}\ell\,m_\ell+2t,\qquad s=\sum_{\ell\ge3}m_\ell\ge1.
\]
Exactly as in~\eqref{eq:term}, this term equals
\[
\frac{M!}{(M-L)!}\,W_1^{\,M-L}\,\frac{(W_2/2)^{t}}{t!}\prod_{\ell\ge3}\frac{(W_\ell/\ell)^{m_\ell}}{m_\ell!},
\]
so $L\le M$, and the block-placement factor $M(M-1)\cdots(M-L+1)$ contributes
\[
\vp_2\bigl(M(M-1)\cdots(M-L+1)\bigr)=(r-1)+\vp_2((L-1)!).
\]
Each cycle of length \(\ell\ge3\) contributes at least
\((\ell-1)-\vp_2(\ell)\ge1\) to $\vp_2(W_\ell/\ell)$ (by the same case split on $\vp_2(\ell)$ as
in the odd case, now with $2^v\ge v+2$ for $v\ge2$, and directly for $v\in\{0,1\}$), while each
cycle of length $2$ contributes \(\vp_2(W_2/2)=\vp_2(5)=0\). Summed with multiplicity, the
length-$\ge3$ cycles contribute at least $s$. The repeated-length denominator
$\prod_{\ell\ge3}m_\ell!$ divides $s!$ by the same multinomial-coefficient argument as before, so
costs at most \(\vp_2(s!)\); the length-$2$ denominator $t!$ costs \(\vp_2(t!)\). Since $s\ge1$, $L=\sum_{\ell\ge3}\ell\,m_\ell+2t\ge
3s+2t\ge2t+3$, so $L-1\ge2t$; as $\vp_2(n!)$ is nondecreasing in $n$ and $2t\ge t$, the term
\(\vp_2((L-1)!)\) is at least \(\vp_2((2t)!)\ge\vp_2(t!)\), which absorbs that cost. Therefore the
total valuation is at least
\[
(r-1)+s-\vp_2(s!)\ge r,
\]
because \(\vp_2(s!)\le s-1\).
\end{proof}

By Lemma~\ref{lem:two-elimination}, only block-cycles of lengths $1$ and $2$
remain modulo $2^r$, so
\begin{equation}\label{eq:S}
a(2^r)\equiv M!\,[x^{M}]\exp\!\left(3x + 5x^2\right) =: S_M,
\end{equation}
using $W_1/1=3$ and $W_2/2=5$. We expand the exponential. The term $5x^2$
supplies $j$ block-cycles of length $2$, occupying $2j$ blocks, while the
remaining $M-2j$ blocks are fixed points of weight $3$. Therefore
\[
S_M = \sum_{j=0}^{\lfloor M/2\rfloor}
        \frac{M!}{j!\,(M-2j)!}\,5^{j}\,3^{\,M-2j}.
\]
Write $M=2^{R}$ with $R=r-1$. We show that for $R\ge2$,
\begin{equation}\label{eq:Sclaim}
S_M\equiv 1\pmod{2^{R+1}},
\end{equation}
which is the same as $S_M\equiv1\pmod{2^{r}}$. We bound the valuation of each
term.

\textit{For $j=0$,} the term equals $3^{M}=3^{2^{R}}$. By the lifting-the-exponent
lemma, for $R\ge1$,
\[
\vp_2\bigl(3^{2^{R}}-1\bigr)
  = \vp_2(3-1)+\vp_2(3+1)+\vp_2(2^R)-1 = 1+2+R-1 = R+2,
\]
so $3^{M}-1$ is divisible by $2^{R+1}$, that is, $3^{M}\equiv1\pmod{2^{R+1}}$.

\textit{For $j\ge1$,} the factors $5^j$ and $3^{M-2j}$ are odd, so the
valuation of the $j$-th term equals the valuation of its coefficient. Using
$M=2^{R}$,
\[
\frac{M!}{(M-2j)!} = M(M-1)\cdots(M-2j+1).
\]
Here the single factor $M$ contributes $\vp_2(M)=R$. For \(1\le a<2^R\), we have
\[
\vp_2(2^R-a)=\vp_2(a).
\]
Therefore
\[
\vp_2\bigl((M-1)\cdots(M-2j+1)\bigr)=\vp_2((2j-1)!).
\]
Subtracting \(\vp_2(j!)\) and
using the identity
\[
\vp_2((2j-1)!)-\vp_2(j!) = (j-1)-\vp_2(j),
\]
which follows from Legendre's formula \(\vp_2(N!)=N-s_2(N)\) and the relation
\(s_2(2j-1)=s_2(j)+\vp_2(j)\),
\[
\vp_2\!\left(\frac{M!}{j!\,(M-2j)!}\right) = R + j - 1 - \vp_2(j).
\]
For $j\ge3$ we have $j-1-\vp_2(j)\ge1$, since $\vp_2(j)\le\log_2 j < j-1$, so
those terms are divisible by $2^{R+1}$. The only terms that need not be
individually divisible by $2^{R+1}$ are $j=1$ and $j=2$, and the condition
$R\ge2$ ensures $M-2j\ge0$, so both are present. For each of these the formula
gives exact valuation $R$, so each term equals $2^{R}$ times an odd number.
Their sum is therefore
\[
2^{R}(\text{odd}) + 2^{R}(\text{odd}) = 2^{R}(\text{odd}+\text{odd})
= 2^{R}(\text{even}),
\]
which is divisible by $2^{R+1}$.

Combining the three observations, the total contribution of all terms with
$j\ge1$ vanishes modulo $2^{R+1}$, while the term $j=0$ is congruent to $1$.
This proves~\eqref{eq:Sclaim}, and so by~\eqref{eq:S},
\[
a(2^r)\equiv1\pmod{2^r}\qquad(r\ge3).
\]
The two small cases follow directly from the sequence, since $a(2)=3\equiv1\pmod2$
and $a(4)=59\equiv3\pmod4$. This completes the proof of
Proposition~\ref{prop:primepower}.
\end{proof}

\begin{remark}
Equivalently, one may choose and pair the $2j$ blocks forming the length-$2$
cycles in $\frac{M!}{2^{j}\,j!\,(M-2j)!}$ ways and leave the remaining $M-2j$
blocks fixed. Giving each length-$2$ block-cycle the weight $W_2=10$ and each
fixed block the weight $W_1=3$ reproduces the term, since
\[\frac{M!}{2^{j}\,j!\,(M-2j)!}\,10^{j}\,3^{M-2j}
= \frac{M!}{j!\,(M-2j)!}\,5^{j}\,3^{M-2j}.\]
\end{remark}

\begin{remark}
The conditions $R\ge1$ in the lifting-the-exponent step and $R\ge2$ in the
analysis of $j=1$ and $j=2$ are necessary. At $R=0$ one has $\vp_2(3-1)=1$,
which is not $R+2$. At $R=1$, that is $M=2$, the full sum is
$S_2 = 3^2 + \tfrac{2!}{1!\,0!}\cdot 5 = 9+10 = 19\equiv3\pmod4$ rather than
$1$. These boundary cases correspond exactly to $a(2)$ and $a(4)$, which are treated separately above.
\end{remark}

\section{Completing the congruences}\label{sec:completing}

\begin{proof}[Proof of Theorem~\ref{thm:main}]
We now combine the prime-power residues of
Proposition~\ref{prop:primepower} into the value of $a(k)\bmod k$ for each class
of $k$, and then apply the product congruence. In each case the passage from
prime-power residues to a residue modulo $k$ is the Chinese remainder theorem,
as in Section~\ref{sec:reduction}. Fix $k\ge1$.

\subsection*{Case 1: \texorpdfstring{$k$}{k} odd}

The case $k=1$ is immediate, since every integer is divisible by $1$, so assume
$k>1$. For each maximal prime-power divisor $q=p^r$ of $k$, the prime $p$ is odd,
and $k/q$ is odd because it is a product of odd prime powers. By
Lemma~\ref{lem:reduction} and Proposition~\ref{prop:primepower},
\[
a(k)\equiv a(q)^{k/q}\equiv(-1)^{k/q}\equiv-1\pmod q,
\]
where the last step uses that $k/q$ is odd. This holds for every maximal
prime-power divisor of $k$, so the Chinese remainder theorem gives
$a(k)\equiv-1\pmod k$. By Proposition~\ref{prop:product},
\[
a(n+k)\equiv a(n)a(k)\equiv -a(n)\pmod k,
\]
so $a(n+k)+a(n)\equiv0\pmod k$, which is $k\mid a(n+k)+a(n)$.

\subsection*{Case 2: \texorpdfstring{$k\equiv0,2,6\pmod8$}{k = 0,2,6 mod 8}}

Let $k=2^{\alpha}m$ with $m$ odd. Since $k$ is even here, $\alpha\ge1$, and the
hypothesis $k\equiv0,2,6\pmod 8$ is equivalent to $\alpha=1$ or $\alpha\ge3$,
that is $\alpha\ne2$, since $\alpha=2$ would give $k\equiv4\pmod8$. For each odd
maximal prime-power divisor $q=p^r$ of $m$, the exponent $k/q$ is even, because
it retains the factor $2^\alpha$, so by Lemma~\ref{lem:reduction} and
Proposition~\ref{prop:primepower},
\[
a(k)\equiv a(q)^{k/q}\equiv(-1)^{k/q}\equiv1\pmod q.
\]
For the part of $k$ that is a power of $2$, there are two subcases. If
$\alpha=1$, then the $2$-component is $2$, and $a(2)\equiv1\pmod2$ by
Proposition~\ref{prop:primepower}, so by Lemma~\ref{lem:reduction},
$a(k)\equiv a(2)^{k/2}\equiv1^{k/2}\equiv1\pmod2$. If
$\alpha\ge3$, then $a(2^\alpha)\equiv1\pmod{2^\alpha}$ by
Proposition~\ref{prop:primepower}, and so, again by Lemma~\ref{lem:reduction},
$a(k)\equiv a(2^\alpha)^{k/2^\alpha}\equiv1\pmod{2^\alpha}$. In every subcase
$a(k)\equiv1$ modulo each maximal prime-power divisor of $k$, so the Chinese
remainder theorem gives $a(k)\equiv1\pmod k$. Hence
\[
a(n+k)\equiv a(n)a(k)\equiv a(n)\pmod k,
\]
which is $k\mid a(n+k)-a(n)$.

\subsection*{Case 3: \texorpdfstring{$k\equiv4\pmod8$}{k = 4 mod 8}}

Let $k=4m$ with $m$ odd, so the part of $k$ that is a power of $2$ is exactly
$4$. Modulo $4$,
\[
a(k)\equiv a(4)^{m}\equiv 3^{m}\equiv 3\pmod 4,
\]
where the last step uses that $m$ is odd and $3^2\equiv1\pmod4$, so $3^m\equiv3$.
For each odd maximal prime-power divisor $q=p^r$ of $m$, the exponent $k/q$ is
even, so as in Case~2, $a(k)\equiv1\pmod q$, and hence $a(k)\equiv1\pmod m$. We
now want the residue modulo $4m=k$ that is congruent to $3$ modulo $4$ and to
$1$ modulo $m$. That residue is
\[
a(k)\equiv 1+\tfrac{k}{2}\pmod k.
\]
To check this, note that $k/2=2m$. Then $1+2m\equiv1\pmod m$, since $m\mid 2m$,
and $1+2m\equiv1+2\equiv3\pmod4$, since $m$ is odd and so $2m\equiv2\pmod 4$. By
the Chinese remainder theorem this residue is unique modulo $k$. By
Proposition~\ref{prop:product},
\[
a(n+k)\equiv a(n)a(k)\equiv a(n)\Bigl(1+\tfrac{k}{2}\Bigr)\pmod k,
\]
so $a(n+k)-a(n)\equiv \tfrac{k}{2}\,a(n)\pmod k$. Multiplying by $2$ gives
$2\bigl(a(n+k)-a(n)\bigr)\equiv k\,a(n)\equiv0\pmod k$, which is
$k\mid2\bigl(a(n+k)-a(n)\bigr)$.

\medskip
This completes the proof of Theorem~\ref{thm:main}.
\end{proof}

\section*{Acknowledgments}

The author thanks the contributors to the OEIS entry for A028342, especially Peter Bala for formulating the congruence conjecture. The author also thanks Gabor Lippner for a helpful discussion of the proof. The author used AI tools for exploratory discussion, computational checking, and drafting assistance. The author independently checked the mathematical arguments, computations, and final manuscript. Theorem~\ref{thm:main} and all of its supporting propositions and lemmas have additionally been formally verified in the Lean~4 proof assistant using Mathlib.\footnote{Formalization available at
\url{https://github.com/ahaankallat/bala-a028342-lean}. The top-level Lean
statement corresponding to Theorem~\ref{thm:main} is
\texttt{main\_theorem} in \texttt{Bala/Complete.lean}.}

\end{document}